\RequirePackage{fix-cm}
\documentclass[smallextended]{svjour3}       
\smartqed  
\usepackage{graphicx}
\usepackage{amsmath}
\usepackage{amssymb}
\usepackage{multirow}
\usepackage{color}
\usepackage{longtable}
\usepackage{array}
\usepackage{url}
\usepackage{enumerate}
\usepackage[11pt]{extsizes}

\newcommand{\lb}{\left[}
\newcommand{\rb}{\right]}
\newcommand{\mb}{\begingroup\setlength\arraycolsep{2pt}\def\arraystretch{1.0}\begin{pmatrix}}
\newcommand{\me}{\end{pmatrix}\endgroup}

\begin{document}

\title{Partial Geometric Designs from Group Actions
}


\author{Jerod Michel         \and
        Qi Wang 
}


\institute{Jerod Michel \at
              Department of Computer Science and Engineering, Southern University of Science and Technology, Shenzhen 518055, China. \\
              \email{michelj@sustc.edu.cn}           
           \and
           Qi Wang \at
              Department of Computer Science and Engineering, Southern University of Science and Technology, Shenzhen 518055, China.\\
              \email{wangqi@sustc.edu.cn}\newline The authors were supported by the National Science Foundation of China under Grant No. 61672015.
}

\date{Received: date / Accepted: date}

\maketitle

\begin{abstract}
In this paper, using group actions, we introduce a new method for constructing partial geometric designs (sometimes referred to as $1\frac{1}{2}$-designs). Using this new method, we construct several infinite families of partial geometric designs by investigating the actions of various linear groups of degree two on certain subsets of $\mathbb{F}_{q}^{2}$. Moreover, by computing the stabilizers of such subsets in various linear groups of degree two, we are also able to construct a new infinite family of balanced incomplete block designs.
\subclass{05B05 \and 05E30}
\end{abstract}

\section{Introduction}\label{sec1}
Combinatorial designs are an important subject of combinatorics intimately related to finite geometry \cite{BET}, \cite{DEM}, \cite{HIR}, \cite{MOOR}, with applications in statistics and experiment design \cite{BOSE}, \cite{FISH}, coding and information theory \cite{ASS}, \cite{CUN}, \cite{GOL}, \cite{HP}, and cryptography \cite{CDR}, \cite{OKS}, \cite{STIN}.
\par
Recent literature shows an increased interest in the study of partial geometric designs. Since their concurrence matrices have three eigenvalues (with one equal to zero), partial geometric designs provide a partial solution to Bailey's well-known question \cite{CAM} concerning when the concurrence matrix of a connected binary equireplicant proper incomplete block design has exactly three eigenvalues. Olmez, in \cite{O}, introduced a method related to difference sets for constructing symmetric partial geometric designs, and in \cite{KN}, Nowak, Olmez and Song generalized this to a method based on difference families. Brouwer, Olmez and Song, in \cite{BOS}, showed that directed strongly regular graphs can be constructed from partial geometric designs. In \cite{O2}, Olmez showed how partial geometric designs can be used to construct plateaued functions, in \cite{O3}, Olmez investigated the link between partial geometric designs and three-weight codes, and in \cite{NOW1}, Nowak and Olmez constructed partial geometric designs with prescribed automorphisms.
\par
One classical method for constructing combinatorial designs is to use group actions \cite[p.~175]{BET}. Interestingly, the difficulty in using this method is not in satisfying the conditions required for the existence of a combinatorial design (in fact, these conditions are often easily satisfied), but in computing the parameters of the design. Some good examples of infinite families of designs obtained using this method can be found in \cite{CAMER}, \cite{LIU}, \cite{LIUTANG}. In this paper, using group actions, we introduce a new method for constructing partial geometric designs and, using this new method, we construct infinite families of partial geometric designs, some of which are new, by investigating the actions of matrix groups of degree two on certain subsets of $\mathbb{F}_{q}^{2}$. Our construction method not only generalizes those of partial geometric difference sets \cite{DAV}, \cite{O} and partial geometric difference families \cite{CHANG}, \cite{MICH00}, \cite{KN} (as our method need not be that of a group acting on itself, but could also be that of a group acting on an arbitrary set), but also generalizes the classical method \cite[p.~175]{BET} for constructing combinatorial designs from group actions. Moreover, by computing the stabilizers of certain subsets of $\mathbb{F}_{q}^{2}$ in various linear groups of degree two, we are also able to construct a new infinite family of balanced incomplete block designs.
\par
The remainder of this paper is organized as follows. Section \ref{sec2} recalls several preliminary concepts that will be used throughout the paper. In Section \ref{sec3} we introduce a new method for constructing partial geometric designs based on group actions. In Section \ref{sec4} we construct infinite families of partial geometric designs using the new method. Section \ref{sec5} concludes the paper as well as discusses some directions for further work.
\section{Preliminaries}\label{sec2}
\subsection{Finite incidence structures}
A (finite) {\it incidence structure} is a triple $(V,\mathcal{B},I)$ such that $V$ is a finite set of elements called {\em points}, $\mathcal{B}$ is a finite set of elements called {\em blocks}, and $I$ ($\subseteq V\times\mathcal{B}$) is an incidence relation between $V$ and $\mathcal{B}$. Since, in the following, all incidence structures $(V,\mathcal{B},I)$ are such that $\mathcal{B}$ is a collection (i.e., a multiset) of nonempty subsets of $V$, and $I$ is given by membership (i.e., a point $p\in V$ and a block $B\in\mathcal{B}$ are incident if and only if $p\in B$), we will denote the incidence structure $(V,\mathcal{B},I)$ simply by $(V,\mathcal{B})$. An incidence structure that has no repeated blocks is called {\em simple}.
\par
A $t$-$(v,k,\lambda)$ {\it design} (or $t$-design, for short) (with $0<t<k<v$) is an incidence structure $(V,\mathcal{B})$ where $V$ is a set of $v$ points and $\mathcal{B}$ is a collection of $k$-subsets of $V$ such that any $t$-subset of $V$ is contained in exactly $\lambda$ blocks \cite{BET}. In the literature, $t$-designs with $t=1$ are often referred to as {\em tactical configurations}, and those with $t=2$ are often referred to as {\em balanced incomplete block designs}. We will denote the number of blocks of an incidence structure by $b$, and the number of blocks containing a given point $u\in V$ by $r_{u}$, and when $(V,\mathcal{B})$ is a tactical configuration, simply by $r$. Then the identities \[
bk=vr,\] and \[r(k-1)=(v-1)\lambda\] restrict the possible sets of parameters of $2$-designs.
\par
Let $(V,\mathcal{B})$ be a tactical configuration where $|V|=v$, each block has cardinality $k$, and each point has replication number $r$. We call a member $(u,B)$ of $V\times \mathcal{B}$ a {\it flag} if $u\in B$, and an {\it antiflag} if $u\notin B$. For each point $u\in V$ and each block $B\in \mathcal{B}$, let $s(u,B)$ denote the number of flags $(w,C)\in V\times \mathcal{B}$ such that $w\in B\setminus\{u\},u\in C$ and $C\neq B$. If there are integers $\alpha$ and $\beta$ such that \[
 s(u,B)=\begin{cases} \alpha, \text{ if } u\notin B, \\
                      \beta, \text{ if } u\in B,\end{cases} \] as $(u,B)$ runs over $V\times\mathcal{B}$, then we say that $(V,\mathcal{B})$ is a {\it partial geometric design} with parameters $(v,k,r;\alpha,\beta)$ \cite{BOSE3}, \cite{NEUM}.
\subsection{Group actions}
Let $V$ be a set of $v$ elements with $v\geq 1$, and $G$ be a permutation group on $V$. For $x\in V$ and $g\in G$, we will denote $g(x)$ by $x^{g}$. For subsets $S\subseteq V$ and $E\subseteq G$, we will use the following abuses of notation:\[
S^{g}=\{x^{g}\mid x\in S\},\quad\quad\quad\quad\quad\quad\quad\quad x^{E}=\{x^{g}\mid g\in E\}, \quad\quad\quad\quad\quad E^{-1}=\{g^{-1}\mid g\in E\}\]\[
S^{E}=\{x^{g}\mid x\in S\text{ and } g\in E\},\quad\quad\quad\quad\left[S\right]^{E}=\{S^{g}\mid g\in E\},\quad\quad\quad\quad x+S=\{x+s\mid s\in S\}.\]
When $E$ is a subgroup of $G$, then $x^{E}$ is called the $E$-{\it orbit} of $x$, and $S^{E}$ is simply the union of $E$-obits of members of $S$. Also, assuming $E$ is a subgroup of $G$, a (right) {\it transversal} of $E$ in $G$, is a subset of $G$ which meets each (right) coset of $E$ in exactly one point. Finally, we define the (setwise) {\it stabilizer}, $G_{S}$, of $S$ in $G$ by $\{g\in G\mid S^{g}=S\}$.
\par
We say $G$ is transitive on $V$ if for each pair $x,y$ of distinct elements of $V$, there exists a member $g\in G$ such that $x^{g}=y$. We say $G$ is $t$-transitive on $V$ if for each pair of ordered $t$-subsets $T,T'\subseteq V$ there exists a member $g\in G$ such that $T^{g}=T'$, and we say that $G$ is $t$-homogeneous if for each pair of (unordered) $t$-subsets $T,T'\subseteq V$ there exists a member of $G$ sending the former to the latter.
\par
The following theorem describes a classical method for constructing $t$-designs by group actions.
\begin{lemma}\label{le1} {\rm \cite{BET}} Let $V$ be a set of $v\geq1$ elements, and $G$ a permutation group on $V$. Let $D$ be a $k$-subset of $V$ with $k\geq 2$. If $G$ is $t$-homogeneous on $V$ (and $k\geq t$) then $(V,\lb D\rb^{G})$ is a $t$-$(v,k,\lambda)$ design with $b$ blocks where \[
\lambda=b\frac{\binom{k}{t}}{\binom{v}{t}}=\frac{|G|}{|G_{D}|}\frac{\binom{k}{t}}{\binom{v}{t}},\]and $G_{D}$ is the set-wise stabilizer of $D$.
\end{lemma}
\subsection{Cyclotomic classes and cyclotomic numbers}
In this section we will need some facts about cyclotomic classes and cyclotomic numbers. Let $q=ef+1$ be a prime power, and $\gamma$ a primitive element of the finite field $\mathbb{F}_{q}$ with $q$ elements. The {\it cyclotomic classes} of order $e$ are given by $D_{i}^{(e,q)}=\gamma^{i}\langle \gamma^{e} \rangle$ for $i=0,1,...,e-1$. The {\it cyclotomic numbers of order $e$} are given by $(i,j)_{e}=|D_{i}^{(e,q)}\cap (D_{j}^{(e,q)}+1)|$. It is obvious that there are at most $e^{2}$ different cyclotomic numbers of order $e$. When it is clear from the context, we will denote $(i,j)_{e}$ simply by $(i,j)$.
\par
We will need to use the cyclotomic numbers of order $2$.
\begin{lemma}\label{le2} {\rm \cite{STO}} For a prime power $q$, if $q \equiv 1$ (mod 4), then the cyclotomic numbers of order two are given by
\begin{eqnarray*}
(0,0) & = & \frac{q-5}{4},                         \\
(0,1) & = & (1,0) = (1,1) = \frac{q-1}{4}.
\end{eqnarray*} If $q \equiv 3$ (mod 4) then the cyclotomic numbers of order two are given by
\begin{eqnarray*}
(0,1) & = & \frac{q+1}{4},                         \\
(0,0) & = & (1,0) = (1,1) = \frac{q-3}{4}.
\end{eqnarray*}
\end{lemma}
\section{Partial Geometric Designs from Group Actions}\label{sec3}
The following provides a method for constructing partial geometric designs from group actions.
\begin{theorem}\label{th1} Let $V$ be a set of $v$ elements where $v\geq 1$, and $G$ be a permutation group acting transitively on $V$. Let $\mathcal{S}=\{S_{1},...,S_{n}\}$ be a family of $k$-subsets ($k\geq 2$) of $V$ such that $S_{i}\notin \left[S_{j}\right]^{G}$ for $i\ne j$. Suppose that for each $i,1\leq i \leq n$, and for each $x\in V$, there are constants $\alpha$ and $\beta$ such that \begin{enumerate}[(i)]
\item $\beta:=\sum_{j=1}^{n}\sum_{x\in B\in \left[S_{j}\right]^{G},B\ne S_{i}}(\left|B\cap S_{i}\right|-1)$ if $x\in S_{i}$, and
\item $\alpha:=\sum_{j=1}^{n}\sum_{x\in B\in \left[S_{j}\right]^{G},B\ne S_{i}}\left|B\cap S_{i}\right|$ if $x\notin S_{i}$.
\end{enumerate} Then $(V,\bigcup_{i=1}^{n}\left[S_{i}\right]^{G})$ is a $(v,k,r;\alpha,\beta)$ partial geometric design with $b$ blocks where $b=\sum_{i=1}^{n}\frac{|G|}{|G_{S_{i}}|}$, and replication number $r=\sum_{i=1}^{n}r_{i}$ where $r_{i}$ is the replication number for $\lb S_{i}\rb^{G}$.

\end{theorem}
\begin{proof} The fact that $(V,\bigcup_{i=1}^{n}\lb S_{i}\rb ^{G})$ is a tactical configuration with replication number $r$ and $b$ blocks follows immediately from Lemma \ref{le1} and the fact that unions of disjoint block sets of tactical configurations, each having the same point set and block cardinality, is again a tactical configuration whose replication number is the sum of the replication numbers of the individual tactical configurations. We need only show that the partial geometric property holds. Fix $(x,B)\in V\times \bigcup_{i=1}^{n}\lb S_{i}\rb^{G}$. Assume that $B=S_{i}^{g'}$ for some $g'\in G$. Then we can, without loss of generality, take $B=S_{i}$. If we suppose that $x\in S_{i}$, then the number of flags $(y,C)\in (V\setminus\{x\})\times (\bigcup_{i=1}^{n}\lb S_{i}\rb ^{G}\setminus\{S_{i}\})$ such that $y\in (S_{i}\setminus\{x\})\cap C$ and $x\in C$ is given by \[
s(x,S_{i})=\sum_{j=1}^{n}\sum_{\substack{x\in C\in\lb S_{j}\rb ^{G},\\C\ne S_{i}}}(\left|C\cap S_{i}\right|-1)=\beta\] and, if we suppose that $x\notin S_{i}$, then the number of flags $(y,C)\in (V\setminus\{x\})\times (\bigcup_{i=1}^{n}\lb S_{i}\rb^{G}\setminus\{S_{i}\})$ such that $y\in S_{i}\cap C$ and $x\in C$ is given by \[
s(x,S_{i})=\sum_{j=1}^{n}\sum_{\substack{x\in C\in \lb S_{j}\rb ^{G},\\C\ne S_{i}}}\left|C\cap S_{i}\right|=\alpha.\]
\end{proof}
The following corollary is an immediate consequence of Theorem \ref{th1}.
\begin{corollary}\label{co1} Let $V$ be a set of $v$ elements where $v\geq 1$, and $G$ be a permutation group acting transitively on $V$. Let $S$ be a $k$-subset ($k\geq 2$) of $V$ such that for each $x\in V$, there are constants $\alpha$ and $\beta$ such that \begin{enumerate}[(i)]
\item $\beta:=\sum_{x\in B\in \lb S\rb^{G},B\ne S}(\left|B\cap S\right|-1)$ if $x\in S$, and
\item $\alpha:=\sum_{x\in B\in \lb S\rb^{G},B\ne S}\left|B\cap S\right|$ if $x\notin S$.
\end{enumerate} Then $(V,\lb S\rb^{G})$ is a $(v,k,r;\alpha,\beta)$ partial geometric design with $b$ blocks where $r=b\frac{k}{v}=\frac{|G|}{|G_{S}|}\frac{k}{v}$.
\end{corollary}
The following points out some feasibility conditions for the parameters of partial geometric designs constructed via Theorem \ref{th1}.
\begin{proposition} Let $V$ be a set of $v$ elements where $v\geq 1$, and $G$ be a permutation group acting transitively on $V$. Let $\mathcal{S}=\{S_{1},...,S_{n}\}$ be a family of $k$-subsets of $V$ satisfying the conditions of Theorem \ref{th1}. Then \begin{enumerate}[(i)]
\item $r\beta=\sum_{i=1}^{n}r_{i}\beta$ is even, and
\item $\frac{r(v-k)}{r+k+\beta-\alpha-1}=\frac{\sum_{i=1}^{n}r_{i}(v-k)}{\sum_{i=1}^{n}r_{i}+k+\beta-\alpha-1}$ is an integer.
\end{enumerate}
\end{proposition}
\begin{proof} Part (i) follows from a remark on page 3 of \cite{VAND}, and part (ii) follows immediately from Lemma 3.12 of \cite{NEUM}.
\end{proof}

\section{Partial Geometric Designs from Linear Group Actions}\label{sec4}
Let $q$ be a prime power, and let ${\rm SL}_{2}(q)$ (resp. ${\rm GL}_{2}(q)$) denote the special linear group (resp. general linear group) of degree two over $\mathbb{F}_{q}$, i.e., the set of all $2\times2$ matrices over $\mathbb{F}_{q}$ with determinant equal to one (resp. with nonzero determinant). Group actions will be given by right multiplication, and the reader should note here that both ${\rm SL}_{2}(q)$ and ${\rm GL}_{2}(q)$ are transitive on $\mathbb{F}_{q}^{2}\setminus\{\bf{0}\}$. For a subgroup $G$ of ${\rm GL}_{2}(q)$, we will use the following notation:\[
G_{(ij)}=\{g\in G\mid g_{ij}=0\}, \quad G_{(+)}=\{g\in G\mid g_{12}=g_{21}=0\}, \quad G_{(-)}=\{g\in G\mid g_{11}=g_{22}=0\}.\]
\subsection{Constructions from ${\rm SL}_{2}(q)$ acting on $\mathbb{F}_{q}^{2}\setminus\{{\bf0}\}$}
For the remainder of this section we define $D(q):=\{(x_{1},x_{2})\in \mathbb{F}_{q}^{2}\mid x_{1}\ne0\ne x_{2}\}\subseteq \mathbb{F}_{q}^{2}$. We will need the following lemmas.
\begin{lemma}\label{le3} The stabilizer of $D(q)$ in ${\rm SL}_{2}(q)$ is given by ${\rm SL}_{2}(q)_{(+)}\cup {\rm SL}_{2}(q)_{(-)}$.
\end{lemma}
\begin{proof} We will denote $D(q)$ simply by $D$. First suppose that $D^{g}=D$ for some $g=\mb a & b\\c & d \me\in {\rm SL}_{2}(q)$. Fix $x=(x_{1},x_{2})\in D$ so that we have \[
x^{g}=(x_{1},x_{2})\mb a & b\\c & d \me=(ax_{1}+bx_{2},cx_{1}+dx_{2})\in D.\] If $g\in {\rm SL}_{2}(q)_{(+)}$ then we can write $g=\mb a & 0\\0 & d \me$ whence $x^{g}=(ax_{1},dx_{2})\in D$ is clear. Similarly, if $g\in {\rm SL}_{2}(q)_{(-)}$ then we can write $g=\mb 0 & b\\ c & 0\me$ whence $x^{g}=(cx_{2},bx_{1})\in D$ is clear. If the entries of $g$ are all nonzero, then we can choose $x_{1},x_{2}\in \mathbb{F}_{q}^{*}$ so that $ax_{1}+cx_{2}=0$ from which comes $x^{g}\notin D$. If $g\in {\rm SL}_{2}(q)_{(12)}$, so $b=0$, then again we can choose $x_{1},x_{2}\in \mathbb{F}_{q}^{*}$ so that $ax_{1}+cx_{2}=0$ and we will have $x^{g}\notin D$. The other cases where only one entry of $g$ is equal to zero can be shown similarly.
\end{proof}
\begin{lemma}\label{le4} $\left|\left\{\mb a & b\\c & d \me\in {\rm SL}_{2}(q)\mid a,b,c,d\ne 0\right\}\right|=(q-1)^{2}(q-2)$.
\end{lemma}
\begin{proof} Notice there are $(q-1)^{2}$ choices for entries $a$ and $d$ of an arbitrary matrix $g=\mb a & b\\c & d \me\in {\rm GL}_{2}(q)$ with no entry equal to zero. For each of these choices there are $(q-1)^{2}-(q-1)$ choices for entries $b$ and $c$ that give $g$ a nonzero determinant. Thus, \[\left|\left\{\mb a & b\\c & d \me\in {\rm GL}_{2}(q)\mid a,b,c,d\ne 0\right\}\right|=(q-1)^{2}((q-1)^{2}-(q-1)).\] Now consider the map \[
\psi:\left\{\mb a & b\\c & d \me\in {\rm GL}_{2}(q)\mid a,b,c,d\ne 0\right\}\rightarrow\left\{\mb a & b\\c & d \me\in {\rm SL}_{2}(q)\mid a,b,c,d\ne 0\right\}\] given by \[
\mb a & b\\c & d \me\mapsto \mb ax^{-1} & bx^{-1}\\c & d \me\]
where $\begingroup\setlength\arraycolsep{2pt}\def\arraystretch{1.0}\begin{vmatrix}a & b\\c & d\end{vmatrix}\endgroup=x$. The map is clearly well-defined. The fiber mapping onto $\mb a'x^{-1} & b'x^{-1}\\c & d \me$ is given by \[\left\{\mb a'' & b''\\c & d \me\in {\rm GL}_{2}(q)\mid a''f=a'x^{-1}, b''f=b'x^{-1},\begingroup\setlength\arraycolsep{2pt}\def\arraystretch{1.0}\begin{vmatrix}a'' & b''\\c & d\end{vmatrix}\endgroup=f\right\}\] which can be written as \[\left\{\mb a'x^{-1}y & b'x^{-1}y\\c & d \me\in {\rm GL}_{2}(q)\mid y\ne 0\right\}.\] Thus, the cardinality of each fiber is $q-1$.
\end{proof}
\begin{theorem}\label{th1.1} Let $q$ be a prime power. Then $(\mathbb{F}_{q}^{2}\setminus\{{\bf0}\},\lb D(q)\rb^{{\rm SL}_{2}(q)})$ is a $(q^{2}-1,(q-1)^{2},q(q-1)/2;\alpha,\beta)$ partial geometric design with $q(q+1)/2$ blocks, where \[
\beta=(q-1)^{3}(q-2)/2 \ \text{ and } \ \alpha=(q-1)^{3}(q-2)/2-q(q-1)/2+1.\]
\end{theorem}
\begin{proof} Again we will denote $D(q)$ simply by $D$. The fact that $(\mathbb{F}_{q}^{2}\setminus\{{\bf0}\},\lb D\rb^{{\rm SL}_{2}(q)})$ is a tactical configuration with replication number $r=q(q-1)/2$ is immediate from Lemma \ref{le1}. Fix $g=\mb a & b\\c & d \me\in {\rm SL}_{2}(q)$ with $a,b,c,d\ne 0$, and let $x=(x_{1},x_{2})\in D$. Suppose that \[
x^{g}=(x_{1},x_{2})\mb a & b\\c & d \me=(ax_{1}+cx_{2},bx_{1}+dx_{2})=(0,w) \ (\ne {\bf 0}).\] Notice that for each $u\in \mathbb{F}_{q}^{*}$ we have that the first coordinate of $(ux_{1},ux_{2})^{g}$ is zero. This gives $q-1$ members $x'\in D$ such $x'^{g}\in\{0\}\times\mathbb{F}_{q}^{*}$.  A similar argument shows that there are another $q-1$ members $x''\in D$ such that $x''^{g}\in\mathbb{F}_{q}^{*}\times\{0\}$. This gives us that $|D^{g}\cap D|=(q-1)^{2}-2(q-1)$.
\par
Now suppose that $g\in {\rm SL}_{2}(q)_{(12)}$. If \[
x^{g}=(x_{1},x_{2})\mb a & 0\\c & d \me=(ax_{1}+cx_{2},dx_{2})=(0,w) \ (\ne {\bf 0}),\] then for each $u\in\mathbb{F}_{q}^{*}$, we have that $(ux_{1},ux_{2})^{g}\in\{0\}\times\mathbb{F}_{q}^{*}$. This implies that $|D^{g}\cap D|=(q-1)^{2}-(q-1)$.
\par
If $x\in (\mathbb{F}_{q}^{2}\setminus\{{\bf 0}\})\setminus D$ then by Lemmas \ref{le3} and \ref{le4} combined with the fact that for any $g \in \left\{\mb a & b\\c & d \me\in {\rm SL}_{2}(q)\mid a,b,c,d\ne 0\right\}$ the translate $D^{g}$ contains both of $\mathbb{F}_{q}\times\{0\}$ and $\{0\}\times\mathbb{F}_{q}$, we have that the number of translates $D^{g}$ containing $x$, where $g \in \left\{\mb a & b\\c & d \me\in {\rm SL}_{2}(q)\mid a,b,c,d\ne 0\right\}$, is given by $(q-1)^{2}(q-2)/|{\rm SL}_{2}(q)_{D}|=(q-1)(q-2)/2$. The $r-(q-1)(q-2)/2=q-1$ remaining blocks containing $x$ are those blocks $D^{g}$ where $g\in {\rm SL}_{2}(q)_{(ij)}$ for $i,j\in\{1,2\}$.
\par
If $x\in D$ then, since ${\rm SL}_{2}(q)_{(-)}^{{\rm SL}_{2}(q)_{(12)}}={\rm SL}_{2}(q)_{(22)}$, we have that the number of blocks $D^{g}$ containing $x$, where $g\in {\rm SL}_{2}(q)_{(12)}$, is given by $|{\rm SL}_{2}(q)_{(12)}|/|{\rm SL}_{2}(q)_{D}|=(q-1)/2$. A similar argument shows that the number of blocks $D^{g}$ containing $x$, where $g\in {\rm SL}_{2}(q)_{(21)}$, is also $(q-1)/2$. Thus, the number of blocks $D^{g}$ containing $x$, where $g\in {\rm SL}_{2}(q)_{(ij)}$ for $i,j\in\{1,2\}$, is given by $q-1$. The remaining $r-(q-1)=(q-1)(q-2)/2$ blocks containing $x$ are those blocks $D^{g}$ where $g$ has no entry equal to zero.
\par
We have thus shown that for each $x\in \mathbb{F}_{q}^{2}\setminus\{{\bf 0}\}$, \begin{eqnarray*}s(x,D) & = & \begin{cases}
\sum_{x\in D^{g}\in\lb D \rb^{{\rm SL}_{2}(q)}\setminus\{D\}}|D^{g}\cap D|,&\text{ if } x\notin D,\\
\sum_{x\in D^{g}\in\lb D \rb^{{\rm SL}_{2}(q)}\setminus\{D\}}|D^{g}\cap D|,&\text{ if } x\in D,
\end{cases}\\
                        & = & \begin{cases}
\lb(q-1)^{2}-2(q-1)\rb\frac{q-1}{2}(q-2)+\lb(q-1)^{2}-(q-1)\rb\lb r-\frac{q-1}{2}(q-2)\rb,&\text{ if } x\notin D,\\
\lb(q-1)^{2}-(q-1)\rb(q-1)+\lb(q-1)^{2}-2(q-1)\rb\lb r-(q-1)\rb,&\text{ if } x\in D,
\end{cases}\\
                        & = & \begin{cases}
(q-1)^{3}(q-2)/2,&\text{ if } x\notin D,\\
(q-1)^{3}(q-2)/2-r+1,&\text{ if } x\in D.
\end{cases}
\end{eqnarray*} The result now follows from Corollary \ref{co1}.
\end{proof}
We give the following examples for small $q$.
\begin{example} (i) The incidence structure $(\mathbb{F}_{5}^{2}\setminus\{{\bf0}\},\lb D(5) \rb^{{\rm SL}_{2}(5)})$ is a $(24,16,10;96,87)$ partial geometric design with $15$ blocks. (ii) The incidence structure $(\mathbb{F}_{7}^{2}\setminus\{{\bf0}\},\lb D(7) \rb^{{\rm SL}_{2}(7)})$ is a $(48,36,21;540,520)$ partial geometric design with 28 blocks.
\end{example}
We also have the following corollary.
\begin{corollary} Let $q$ be a prime power. Then $(\mathbb{F}_{q}^{2}\setminus\{{\bf0}\},\lb D(q)\rb^{{\rm GL}_{2}(q)})$ is a $(q^{2}-1,(q-1)^{2},q\frac{q-1}{2};\alpha,\beta)$ partial geometric design with $q(q+1)/2$ blocks, where \[
\beta=(q-1)^{3}(q-2)/2 \ \text{ and } \ \alpha=(q-1)^{3}(q-2)/2-q(q-1)/2+1.\]
\end{corollary}
\begin{proof} An argument similar to that given in the proof of Lemma \ref{le2} shows that ${\rm GL}_{2}(q)_{D(q)}={\rm GL}_{2}(q)_{(+)}\cup{\rm GL}_{2}(q)_{(-)}$. The rest of the details, which we leave to the reader, are similar to those of the proof of Theorem \ref{th1.1}.
\end{proof}
We remark that one can construct partial partial geometric designs having the same parameters as those obtained by Theorem \ref{th1.1}, but with repeated blocks, by taking the kronecker product between the $(q-1)\times q$ all-one matrix and the matrix $\mb J-I & J-I\me^{T}$, where $J$ and $I$ are the $(q+1)/2\times (q+1)/2$ all-one matrix and identity matrix, respectively (see Theorem 6 of \cite{KN1}). It is easy to see, however, that the partial geometric designs so constructed are not simple, and therefore are different from those obtained by Theorem \ref{th1.1}, which are simple.
\subsection{Constructions from subgroups of ${\rm SL}_{2}(q)$ acting on $\mathbb{F}_{q}^{*}\times\mathbb{F}_{q}$}\label{ssec4.1}
For the remainder of this section we define $D(q)=\{(\gamma^{i},\gamma^{j})\in \mathbb{F}_{q}^{2}\mid i+j\equiv 0 \ ({\rm mod} \ 2)\}$ and $D'(q)=\{(\gamma^{i},\gamma^{j})\in \mathbb{F}_{q}^{2}\mid i+j\equiv 1 \ ({\rm mod} \ 2)\}$. Note that $D(q)$ is a special case of standard cyclotomy \cite{GUST}, often referred to as the twin prime cyclotomy, and has been extensively studied \cite{CUN}, \cite{ZHANG} in relation to other combinatorial objects such as almost difference sets. Also, we will denote the $l$-th cyclotomic class $D_{l}^{(2,q)}$ of order two simply by $D_{l}$. We will need the following lemmas.
\begin{lemma}\label{le44} Let $q$ be a prime power. The stabilizer of both $D(q)$ and $D'(q)$ in ${\rm SL}_{2}(q)$ is given by \[\begin{cases}
{\rm SL}_{2}(q)_{(+)}\cup {\rm SL}_{2}(q)_{(-)},&\text{ if } q\equiv 1 \ ({\rm mod} \ 4),\\
{\rm SL}_{2}(q)_{(+)},&\text{ if } q\equiv 3 \ ({\rm mod} \ 4).\end{cases}\]
\end{lemma}
\begin{proof} Denote $D(q)$ and $D'(q)$ by $D$ and $D'$, respectively. We will compute the stabilizer for $D$. That for $D'$ can be computed similarly. First suppose that $D^{g}=D$ for some $g=\mb a & b\\c & d \me\in {\rm SL}(2,q)$. Fix $x=(\gamma^{i},\gamma^{j})\in D$ so that we have \[
x^{g}=(\gamma^{i},\gamma^{j})\mb a & b\\c & d \me=(a\gamma^{i}+b\gamma^{j},c\gamma^{i}+d\gamma^{i})\in D.\] If $g\in {\rm SL}_{2}(q)_{(+)}$ then we can write $g=\mb \gamma^{i_{1}} & 0\\0 & \gamma^{i_{2}} \me$ whence $x^{g}=(\gamma^{i+i_{1}},\gamma^{j+i_{2}})\in D$ is clear. Similarly, if $g\in {\rm SL}_{2}(q)_{(-)}$ then we can write $g=\mb 0 & \gamma^{i_{1}}\\ \gamma^{i_{2}} & 0\me$ whence $x^{g}=(\gamma^{i+i_{1}},\gamma^{j+i_{2}})\in D$ is clear if $q\equiv 1 \ ({\rm mod} \ 4)$ since $(q-1)/2\equiv 0 \ ({\rm mod} \ 2)$. If $q\equiv 3 \ ({\rm mod} \ 4)$ then $x^{g}\notin D$ since $(q-1)/2\equiv 1 \ ({\rm mod} \ 2)$.
\par
Now let $g\in \left\{\mb a & b\\c & d \me\in{\rm SL}_{2}(q)\mid a,b,c,d\ne 0 \right\}$ and suppose that \[(1,1)^{g}=(1,1)\mb \gamma^{i_{1}} & \gamma^{i_{2}}\\ \gamma^{i_{3}} & \gamma^{i_{4}} \me=(\gamma^{i_{1}}(1+\gamma^{i_{3}-i_{1}}),\gamma^{i_{2}}(1+\gamma^{i_{4}-i_{2}}))\in D,\]i.e., suppose that $i_{1}+i_{2}+\kappa_{1}+\kappa_{2}\equiv 0 \ ({\rm mod} \ 2)$ where $\kappa_{1}={\rm log}_{\gamma}(1+\gamma^{i_{3}-i_{1}})$ and $\kappa_{2}={\rm log}_{\gamma}(1+\gamma^{i_{4}-i_{2}})$. If $g\in {\rm SL}_{2}(q)_{D}$ then we must have $(1,\gamma^{j})^{g}\in D$ for all $j\equiv 0 \ ({\rm mod} \ 2)$. We will show that $g\notin {\rm SL}_{2}(q)_{D}$ by considering the moduli of $\kappa_{1}$ and $\kappa_{2}$.
\par
Suppose that $\kappa_{1}\equiv\kappa_{2}\equiv 0 \ ({\rm mod} \ 2)$. We will show that there exists a $j\equiv 0 \ ({\rm mod} \ 2)$ such that $1+\gamma^{i_{3}-i_{1}+j}\in D_{l}$ and $1+\gamma^{i_{3}-i_{1}+j}\in D_{l+1}$. Suppose there is no such $j$, i.e., for all $j\equiv 0 \ ({\rm mod} \ 2)$ we have $\{\gamma^{-j}+\gamma^{i_{3}-i_{1}},\gamma^{-j}+\gamma^{i_{4}-i_{2}}\}\subseteq D_{0}$ or $\subseteq D_{1}$. Say $x=\gamma^{i_{3}-i_{1}}$ and $y=\gamma^{i_{4}-i_{2}}$. Then for each $u\in D_{0}$ we have $u+D_{0}\subseteq D_{0}$ or $u+D_{0}\subseteq D_{1}$. Say $u+\{x,y\}\subseteq D_{0}$ for all $u\in A$ and $u+\{x,y\}\subseteq D_{1}$ for all $u\in B$ (so $A\cup B=D_{0}$). Then we have that $x-y$ appears as a difference of members of $D_{0}$ exactly $|A|$ times, and as a difference of members of $D_{1}$ exactly $|B|$ times. If $q\equiv 3 \ ({\rm mod} \ 4)$ then by Theorem 5.4 of \cite{CUN} we have $|D_{0}|=|A|+|B|=(q-3)/2$, a contradiction. If $q\equiv 1 \ ({\rm mod} \ 4)$ then we have \begin{equation}\label{eq1}
|(x-y)^{-1}D_{0}\cap(1+(x-y)^{-1}D_{0})|=\begin{cases}(0,0)_{2},&\text{ if } y-x\in D_{0},\\
(1,1)_{2},& \text{ otherwise,}\end{cases}\end{equation} and \begin{equation}\label{eq2}
|(x-y)^{-1}D_{1}\cap(1+(x-y)^{-1}D_{1})|=\begin{cases}(1,1)_{2},&\text{ if } y-x\in D_{0},\\
(0,0)_{2},& \text{ otherwise.}\end{cases}\end{equation} Thus, by Theorem 5.4 of \cite{CUN} we have $|A|,|B|\in\{(q-5)/4,(q-1)/4\}$, and by (\ref{eq1}) and (\ref{eq2}) we have $|A|\ne|B|$. Thus, again we have $|D_{0}|=(q-3)/2$, a contradiction.
\par
Now suppose that one of $\kappa_{1},\kappa_{2}$ is $\equiv 0 \ ({\rm mod} \ 2)$ and the other is not. Now we want to choose a $j\equiv 0 \ ({\rm mod} \ 2)$ such that $\{1+\gamma^{i_{3}-i_{1}+j},1+\gamma^{i_{4}-i_{2}+j}\}\subseteq D_{0}$ or $\subseteq D_{1}$. Suppose that for all $j\equiv 0 \ ({\rm mod} \ 2)$ we have $1+\gamma^{i_{3}-i_{1}+j}\in D_{l}$ and $1+\gamma^{i_{3}-i_{1}+j}\in D_{l+1}$. Then for all $u\in D_{0}$ we have $u+\{x,y\}\subsetneq D_{0}$ and $\subsetneq D_{1}$ where $x=\gamma^{i_{3}-i_{1}}$ and $y=\gamma^{i_{4}-i_{2}}$. But this would imply that $u+\{x,y\}\subsetneq D_{0}$ or $\subsetneq D_{1}$ for all $u\in D_{1}$, and we have seen from the previous case that the number of members $u\in \mathbb{F}_{q}$ such that $u+\{x,y\}\subsetneq D_{0}$ or $\subsetneq D_{1}$ is not equal to $(q-1)/2=|D_{1}|$, another contradiction. Thus, no such matrix can fix $D$.
\par
Now let $g\in {\rm SL}_{2}(q)_{(12)}$ and suppose that \[(1,1)^{g}=(1,1)\mb \gamma^{i_{1}} & 0\\ \gamma^{i_{3}} & \gamma^{i_{4}} \me=(\gamma^{i_{1}}(1+\gamma^{i_{3}-i_{1}}),\gamma^{i_{4}})\in D,\] i.e. suppose that $1+\gamma^{i_{3}-i_{1}}\in D_{0}$ ($i_{1}+i_{4}=0$ since $g\in{\rm SL}_{2}(q)$). Again, if $g\in {\rm SL}_{2}(q)_{D}$, then $(1,\gamma^{j})^{g}\in D$ for all $j\equiv 0 \ ({\rm mod} \ 2)$. But notice there always exists a $j\equiv 0 \ ({\rm mod} \ 2)$ such that $1+\gamma^{i_{3}-i_{1}+j}\in D_{1}$, otherwise we would have $\gamma^{-j}+\gamma^{i_{3}-i_{1}}\in D_{0}$ for all $j\equiv 0 \ ({\rm mod} \ 2)$, or, in other words, letting $x=\gamma^{i_{3}-i_{1}}$, we would have $u+x\in D_{0}$ for all $u\in D_{0}$, which implies that $D_{0}=D_{0}+x$, a contradiction. Thus, we must have $g\notin {\rm SL}_{2}(q)_{D}$. The other cases where exactly one entry of $g$ is equal to zero can be shown similarly.
\end{proof}
The proof of the following lemma is straightforward, and so is omitted.
\begin{lemma}\label{le5} The following relations hold: \[{\rm SL}_{2}(q)_{(12)}^{{\rm SL}_{2}(q)_{(11)}}={\rm SL}_{2}(q)_{(11)},\quad\quad\quad\quad{\rm SL}_{2}(q)_{(22)}^{{\rm SL}_{2}(q)_{(11)}}={\rm SL}_{2}(q)_{(21)},\]\[
{\rm SL}_{2}(q)_{(21)}^{{\rm SL}_{2}(q)_{(22)}}={\rm SL}_{2}(q)_{(22)},\quad\quad\quad\quad{\rm SL}_{2}(q)_{(11)}^{{\rm SL}_{2}(q)_{(22)}}={\rm SL}_{2}(q)_{(12)}.\]
\end{lemma}
\begin{lemma}\label{le6} Let $q\equiv 3 \ ({\rm mod} \ 4)$ be a prime power, and let $g\in{\rm SL}_{2}(q)\setminus{\rm SL}_{2}(q)_{D(q)}$. Then \[
|D(q)\cap D(q)^{g}|=\begin{cases} \frac{q+1}{4}(q-1),& \text{if } g\in {\rm SL}_{2}(q)_{(11)}\text{ and }\log_{\gamma}(g_{21}g_{22})\equiv1 \ ({\rm mod} \ 2),\\
                                               & \text{ or if } g\in {\rm SL}_{2}(q)_{(22)}\text{ and }\log_{\gamma}(g_{11}g_{12})\equiv1 \ ({\rm mod} \ 2),\\
                            \frac{q-3}{4}(q-1),& \text{ otherwise.}
\end{cases}\]
\end{lemma}
\begin{remark}\label{re1} One can easily compute ${\rm SL}_{2}(q)_{(ij)}^{{\rm SL}_{2}(q)_{(i'j')}}$ for all $i,i',j,j'\in\{1,2\}$ by using Lemma \ref{le5} combined with the fact that ${\rm SL}_{2}(q)_{(ii)}^{-1}={\rm SL}_{2}(q)_{(jj)}$ whenever $i\ne j$. Also, using Lemmas \ref{le5} and \ref{le6} together with the above mentioned fact, one can readily compute $|S_{1}\cap S_{2}^{g}|$ for all $S_{1},S_{2}\in \{D(q),D'(q)\}$ and all $g\in {\rm SL}_{2}(q)$.
\end{remark}
It is possible to obtain partial geometric designs with new parameters using group actions, as is shown in the following remark.
\begin{remark}\label{re2} According to our numerical results, the incidence structure $(\mathbb{F}_{q}^{2}\setminus\{{\bf 0}\},\lb D(q)\rb^{{\rm SL}_{2}(q)})$ is a partial geometric design whose underlying tactical configuration has parmeters $(v,b,k,r)=(q^{2}-1,q(q+1)/2,(q-1)^{2}/2,q)$ if $q\equiv 1 \ ({\rm mod} \ 4)$, and parameters $(q^{2}-1,q(q+1),(q-1)^{2}/2,q(q-1)/2)$ if $q\equiv 3 \ ({\rm mod} \ 4)$. These parameters are new, which indicates that this method is useful. However, to compute the quantities $|S_{1}\cap S_{2}^{g}|$ for $S_{1},S_{2}\in \{D(q),D'(q)\}$, where $g\in \left\{\mb a & b\\c & d \me\in {\rm SL}_{2}(q)\mid a,b,c,d\ne 0\right\}$, seems difficult since it requires the triple intersection numbers, i.e. the values $|D_{i}^{(2,q)}\cap(D_{j}^{(2,q)}+1)\cap(D_{l}^{(2,q)}+\omega)|$ (where $\omega\ne 1$), which are still unknown. The first few parameter sets of these partial geometric designs, for $q=5,7$ and $11$ respectively, are: $(24,8,5;12,8)$, $(48,18,21;160,108)$, and $(120,50,55;1246,1000)$.
\end{remark}
\noindent{\it Proof of Lemma \ref{le6}.} Denote $D(q)$ resp. $D'(q)$ by $D$ resp. $D'$. We will show two representative cases. The other cases are almost identical.
\par
Let $g=\mb \gamma^{i_{1}} & 0\\ \gamma^{i_{3}} & \gamma^{i_{4}} \me\in {\rm SL}_{2}(q)_{(12)}$ and $x=(\gamma^{i},\gamma^{j})\in D$. Then \begin{align*}
& \quad\quad x^{g} = (\gamma^{i},\gamma^{j})\mb \gamma^{i_{1}} & 0\\ \gamma^{i_{3}} & \gamma^{i_{4}} \me=(\gamma^{i+i_{1}}(1+\gamma^{i_{3}-i_{1}+j-i}),\gamma^{j+i_{4}})\in D\\
&\Leftrightarrow i+j+\kappa\equiv0 \ ({\rm mod} \ 2) \ (\text{where } \kappa=\log_{\gamma}(1+\gamma^{i_{3}-i_{1}+j-i})) &&(\text{since }i_{1}+i_{4}=0)\\
&\Leftrightarrow \kappa \equiv0 \ ({\rm mod} \ 2)&&(\text{since } x\in D)\\
&\Leftrightarrow \gamma^{i-j}\in D_{0}-\gamma^{i_{3}-i_{1}}.
\end{align*} Then by Lemma \ref{le2} we have \begin{eqnarray*} |D^{g}\cap D| & = & (q-1)|\gamma^{i_{1}-i_{3}}D_{0}\cap(1+\gamma^{i_{1}-i_{3}}D_{0})|\\
                                                  & = & (q-1)\begin{cases} (0,0)_{2},&\text{ if } i_{1}-i_{3}\equiv0 \ ({\rm mod} \ 2),\\
                                                                           (1,1)_{2},&\text{ otherwise,}\end{cases}\\
                                                  & = & \frac{q-3}{4}(q-1).\end{eqnarray*}
Now suppose that $g=\mb 0 & \gamma^{i_{2}}\\ \gamma^{i_{3}} & \gamma^{i_{4}} \me\in {\rm SL}_{2}(q)_{(11)}$. Then \begin{align*}
& \quad\quad x^{g} = (\gamma^{i},\gamma^{j})\mb 0 & \gamma^{i_{2}}\\ \gamma^{i_{3}} & \gamma^{i_{4}} \me=(\gamma^{j+i_{3}},\gamma^{j+i_{4}}(1+\gamma^{i_{2}-i_{4}+i-j}))\in D\\
&\Leftrightarrow i_{3}+i_{4}+\kappa\equiv0 \ ({\rm mod} \ 2) \ (\text{where } \kappa=\log_{\gamma}(1+\gamma^{i_{2}-i_{4}+i-j})).
\end{align*} If $i_{3}+i_{4}\equiv0 \ ({\rm mod} \ 2)$ then $\kappa\equiv0 \ ({\rm mod} \ 2)$ if and only if $\gamma^{j-i}\in D_{0}-\gamma^{i_{2}-i_{4}}$ whence by Lemma \ref{le2} \[
|D^{g}\cap D|=(q-1)|\gamma^{i_{2}-i_{4}}D_{0}\cap(1+\gamma^{i_{2}-i_{4}}D_{0})|=\frac{q-3}{4}(q-1).\] If $i_{3}+i_{4}\equiv1 \ ({\rm mod} \ 2)$ then $\kappa\equiv1 \ ({\rm mod} \ 2)$ if and only if $\gamma^{j-i}\in D_{1}-\gamma^{i_{2}-i_{4}}$ whence by Lemma \ref{le2} \begin{align*} |D^{g}\cap D|
&=(q-1)|\gamma^{i_{2}-i_{4}}D_{1}\cap(1+\gamma^{i_{2}-i_{4}}D_{0})|\\
&=(q-1)|D_{1}\cap(1+D_{0})|&&(\text{since }i_{2}=(q-1)/2-i_{3})\\
&=\frac{q+1}{4}(q-1).
\end{align*}\qed
We are now ready to give our second construction of partial geometric designs.
\begin{theorem}\label{th2} Let $q\equiv 3 \ ({\rm mod} \ 4)$ be a prime power. Then $(\mathbb{F}_{q}^{*}\times\mathbb{F}_{q},\lb D(q)\rb^{{\rm SL}_{2}(q)_{(12)}}\cup\lb D'(q)\rb^{{\rm SL}_{2}(q)_{(12)}})$ and $(\mathbb{F}_{q}\times\mathbb{F}_{q}^{*},\lb D(q)\rb^{{\rm SL}_{2}(q)_{(21)}}\cup\lb D'(q)\rb^{{\rm SL}_{2}(q)_{(21)}})$ are both $(q(q-1),(q-1)^{2}/2,q-1;\rho,\rho+1)$ partial geometric designs each having $2q$ blocks, where $\rho=(q-1)^{2}(q-3)/4$.
\end{theorem}
\begin{proof} Denote $D(q)$ and $D'(q)$ by $D$ and $D'$, respectively. We will show only the former. The latter follows by symmetry. It is clear that ${\rm SL}_{2}(q)_{(12)}$ is transitive on $\mathbb{F}_{q}^{*}\times\mathbb{F}_{q}$. That $(\mathbb{F}_{q}^{*}\times\mathbb{F}_{q},\lb D\rb^{{\rm SL}_{2}(q)_{(12)}}\cup\lb D'\rb^{{\rm SL}_{2}(q)_{(12)}})$ is a tactical configuration with $2q$ blocks replication number $q-1$ follows from the fact that each of $(\mathbb{F}_{q}^{*}\times\mathbb{F}_{q},\lb D\rb^{{\rm SL}_{2}(q)_{(12)}})$ and $(\mathbb{F}_{q}^{*}\times\mathbb{F}_{q},\lb D'\rb^{{\rm SL}_{2}(q)_{(12)}})$ is tactical configuration (whose block sets are disjoint) with $q$ blocks and replication number $(q-1)/2$ by Lemma \ref{le1}. We need only show that the partial geometric property holds.
\par
Let $x\in \mathbb{F}_{q}^{*}\times\mathbb{F}_{q}$. Suppose $x\notin D$. We count the number of blocks in $\lb D\rb^{{\rm SL}_{2}(q)_{(12)}}\cup\lb D'\rb^{{\rm SL}_{2}(q)_{(12)}}\setminus\{D\}$ containing $x$ such that $|D^{g}\cap D|=(q-3)(q-1)/4$. If $g\in {\rm SL}_{2}(q)_{(12)}$ and $x\in D^{g}$ then, by Lemma \ref{le6}, $|D^{g}\cap D|=(q-3)(q-1)/4$. If $x\in D'^{g}$ then by Lemmas \ref{le5} and \ref{le6} we have \[|D'^{g}\cap D|=|D^{g'}\cap D|=\begin{cases}\frac{q-3}{4}(q-1),&\text{if }\log_{\gamma}(g'_{11}g'_{12})\equiv1 \ ({\rm mod} \ 2),\\
                                                                                                      \frac{q+1}{4}(q-1),&\text{otherwise.}\end{cases}\]
where $g'\in{\rm SL}_{2}(q)_{(22)}$. Also notice that if $x\notin D$ then either $x\in\{0\}\times\mathbb{F}_{q}^{*}$, or $x\in D'$. If the former holds, then $x=(0,\gamma^{j})$ for some $j$, and we have, for $g\in {\rm SL}_{2}(q)_{(22)}$ and $y=(\gamma^{l},\gamma^{l'})\in D$, that \begin{align*}
&\quad\quad y^{g}=(\gamma^{l},\gamma^{l'})\mb \gamma^{i_{1}} & \gamma^{i_{2}}\\ \gamma^{i_{3}} & 0 \me=(\gamma^{i_{1}+l}+\gamma^{i_{3}+l'},\gamma^{i_{2}+l})=(0,\gamma^{j})\\
&\Leftrightarrow i_{3}+l'=(q-1)/2+i_{1}+l\text{ and }i_{2}+l=j\\
&\Leftrightarrow i_{3}+i_{1}\equiv1 \ ({\rm mod} \ 2)\\
&\Leftrightarrow i_{1}+i_{2}\equiv0 \ ({\rm mod} \ 2).&&(\text{since }i_{3}=(q-1)/2-i_{2})\end{align*}
Thus, if $x\in\{0\}\times\mathbb{F}_{q}^{*}$, then $|D^{g}\cap D|=|D'^{g}\cap D|=(q-3)(q-1)/4$ for all $g\in {\rm SL}_{2}(q)_{(12)}$ such that $x\in D^{g}$. Then we have \begin{eqnarray*}
\sum_{\substack{x\in B\in\lb D\rb^{{\rm SL}_{2}(q)_{(12)}},\\B\ne D}}|B\cap D|+\sum_{\substack{x\in B\in\lb D'\rb^{{\rm SL}_{2}(q)_{(12)}},\\B\ne D}}|B\cap D|
& = & \frac{(q-1)^{2}}{2}\cdot\frac{q-3}{4}+\frac{(q-1)^{2}}{2}\cdot\frac{q-3}{4}\\
& = & (q-1)^{2}\frac{q-3}{4}.\end{eqnarray*}If the latter holds, then $x=(\gamma^{i},\gamma^{j})$ where $i+j\equiv1 \ ({\rm mod} \ 2)$. Let $g'\in {\rm SL}_{2}(q)_{(22)}$ with $g'_{11}g'_{12}\equiv0 \ ({\rm mod} \ 2)$, and let $y=(\gamma^{l},\gamma^{l'})\in D$. We can, without loss of generality, assume that $l=0$. Then we have that \begin{equation}\label{eq2}
y^{g'}=(1,\gamma^{l'})\mb \gamma^{i_{1}} & \gamma^{i_{2}}\\ \gamma^{i_{3}} & 0 \me=(\gamma^{i_{1}}+\gamma^{i_{3}+l'},\gamma^{i_{2}})=(\gamma^{i},\gamma^{j})\Leftrightarrow \gamma^{i_{1}-i_{3}}+\gamma^{l'}=\gamma^{i-i_{3}}.\end{equation} With $x$ fixed, if we let $y$ run over $D$, and let $g'$ run over $\{g\in {\rm SL}_{2}(q)_{(22)}\mid \log_{\gamma}(g_{11}g_{12})\equiv0 \ ({\rm mod} \ 2)\}$, we can see that the number of solutions to (\ref{eq2}), since $i_{1}+(q-1)/2-i_{3}=i_{1}+i_{2}\equiv0 \ ({\rm mod} \ 2)$, is given by the number of pairs $(u,u')\in D_{1}\times D_{0}$ such that $u+u'=\gamma^{i-i_{3}}$, which (using Lemma \ref{le2}) is equal to $|(\gamma^{i_{3}-i}D_{0}+1)\cap \gamma^{i_{3}-i}D_{0}|=(q-3)/4$. Thus, if $x\in D'$, then \[
|D'^{g}\cap D|=|D^{g'}\cap D|=\begin{cases}\frac{q-3}{4}(q-1),& \text{if } g'\in \{g\in {\rm SL}_{2}(q)_{(22)}\mid \log_{\gamma}(g_{11}g_{12})\equiv0 \ ({\rm mod} \ 2)\}\text{ and }x\in D^{g'},\\
                                            \frac{q+1}{4}(q-1),& \text{if }g'\in \{g\in {\rm SL}_{2}(q)_{(22)}\mid \log_{\gamma}(g_{11}g_{12})\equiv1 \ ({\rm mod} \ 2)\}\text{ and }x\in D^{g'},\end{cases}
\]whence, noting that $D\cap D'=\emptyset$, we have\begin{eqnarray}\label{eq4}
\lefteqn{\sum_{\substack{x\in B\in\lb D\rb^{{\rm SL}_{2}(q)_{(12)}},\\B\ne D}}|B\cap D|+\sum_{\substack{x\in B\in\lb D'\rb^{{\rm SL}_{2}(q)_{(12)}},\\B\ne D}}|B\cap D|}\\\nonumber
& = & \frac{q-1}{2}\cdot\frac{q-3}{4}\cdot(q-1)+\frac{q-3}{4}\cdot\frac{q+1}{4}\cdot(q-1)+\frac{q-3}{4}\cdot\frac{q-3}{4}\cdot(q-1)\\\nonumber
& = & (q-1)^{2}\cdot\frac{q-3}{4}.\nonumber
\end{eqnarray} Also, if $x\in D$, we can also deduce that \begin{eqnarray}\label{eq5}
\lefteqn{\sum_{\substack{x\in B\in\lb D\rb^{{\rm SL}_{2}(q)_{(12)}},\\B\ne D}}(|B\cap D|-1)+\sum_{\substack{x\in B\in\lb D'\rb^{{\rm SL}_{2}(q)_{(12)}},\\B\ne D}}(|B\cap D|-1)}\\\nonumber
& = & \frac{q-3}{2}\cdot\frac{q-3}{4}\cdot(q-1)+\frac{q+1}{4}\cdot\frac{q+1}{4}\cdot(q-1)+\frac{q-3}{4}\cdot\frac{q-3}{4}\cdot(q-1)\\\nonumber
& = & (q-1)^{2}\cdot\frac{q-3}{4}+1.\nonumber
\end{eqnarray} A similar argument shows that (\ref{eq4}) and (\ref{eq5}) still hold after substituting $D'$ for $D$ in the summands. The result now follows readily from Theorem \ref{th1}.
\end{proof}
We give the following examples for small $q$.
\begin{example} (i) The incidence structure $(\mathbb{F}_{7}^{*}\times\mathbb{F}_{7},\lb D(7) \rb^{{\rm SL}_{2}(7)_{(12)}}\cup\lb D'(7) \rb^{{\rm SL}_{2}(7)_{(12)}})$ is a $(42,18,6;36,37)$ partial geometric design with $14$ blocks. (ii) The incidence structure $(\mathbb{F}_{11}^{*}\times\mathbb{F}_{11},\lb D(11) \rb^{{\rm SL}_{2}(11)_{(12)}}\cup\lb D'(11) \rb^{{\rm SL}_{2}(11)_{(12)}})$ is a $(110,36,10;200,201)$ partial geometric design with 22 blocks.
\end{example}
We have the following corollary which shows that the parameters of the partial geometric designs obtained after replacing ${\rm SL}_{2}(q)$ by ${\rm GL}_{2}(q)$ in Theorem \ref{th2} are unchanged. Its proof involves tracking the steps of that for Theorem \ref{th2}, and so is briefly summarized.
\begin{corollary} Let $q\equiv 3 \ ({\rm mod} \ 4)$ be a prime power. Then $(\mathbb{F}_{q}^{*}\times\mathbb{F}_{q},\lb D(q)\rb^{{\rm GL}_{2}(q)_{(12)}}\cup\lb D'(q)\rb^{{\rm GL}_{2}(q)_{(12)}})$ and $(\mathbb{F}_{q}\times\mathbb{F}_{q}^{*},\lb D(q)\rb^{{\rm GL}_{2}(q)_{(21)}}\cup\lb D'(q)\rb^{{\rm GL}_{2}(q)_{(21)}})$ are both $(q(q-1),(q-1)^{2}/2,q-1;\rho+1,\rho)$ partial geometric designs, each having $2q$ blocks, where $\rho=(q-1)^{2}(q-3)/4$.
\end{corollary}
\begin{proof} Again denote $D(q)$ and $D'(q)$ by $D$ and $D'$, respectively. We first show that the stabilizer of both $D$ and $D'$ in ${\rm GL}_{2}(q)_{D'}$ is given by \[\begin{cases}
{\rm GL}_{2}(q)_{(+)}\cup {\rm GL}_{2}(q)_{(-)},&\text{ if } q\equiv 1 \ ({\rm mod} \ 4),\\
{\rm GL}_{2}(q)_{(+)},&\text{ if } q\equiv 3 \ ({\rm mod} \ 4),\end{cases}\] the result then follows from an analogous series of arguments similar to those made in the proof of Theorem \ref{th2} (including ${\rm GL}_{2}(q)$-analogues of Lemmas \ref{le5} and \ref{le6}) after making the appropriate adjustments required when ${\rm SL}_{2}(q)$ is replaced by ${\rm GL}_{2}(q)$.
\par
We compute the stabilizer for $D$. That for $D'$ can be computed in a similar way. First suppose that $D^{g}=D$ for some $g=\mb a & b\\c & d \me\in {\rm GL}(2,q)$. Fix $x=(\gamma^{i},\gamma^{j})\in D$ so that we have \[
x^{g}=(\gamma^{i},\gamma^{j})\mb a & b\\c & d \me=(a\gamma^{i}+b\gamma^{j},c\gamma^{i}+d\gamma^{i})\in D.\] If $g\in {\rm GL}_{2}(q)_{(+)}$ then we can write $g=\mb \gamma^{i_{1}} & 0\\0 & \gamma^{i_{2}} \me$ whence $x^{g}=(\gamma^{i+i_{1}},\gamma^{j+i_{2}})\in D$ if and only if $i_{1}+i_{2}\equiv 0 \ ({\rm mod} \ 2)$. Similarly, if $g\in {\rm SL}_{2}(q)_{(-)}$ then we can write $g=\mb 0 & \gamma^{i_{1}}\\ \gamma^{i_{2}} & 0\me$ whence $x^{g}=(\gamma^{i+i_{1}},\gamma^{j+i_{2}})\in D$ if and only if $i_{1}+i_{2}\equiv 0 \ ({\rm mod} \ 2)$.
\par
The cases where $g\in \left\{\mb a & b\\c & d \me\in{\rm SL}_{2}(q)\mid a,b,c,d\ne 0 \right\}$, and where $g\in {\rm GL}_{2}(q)_{(i'j')}$ (for some $i',j'\in \{1,2\}$) are identical to those shown in the proof of Lemma \ref{le44} where the stabilizer of $D$ in ${\rm SL}_{2}(q)$ is computed. The result then follows from an analogous series of arguments similar to those made in the proof of Theorem \ref{th2} after making the appropriate adjustments required when ${\rm SL}_{2}(q)$ is replaced by ${\rm GL}_{2}(q)$.
\end{proof}
We note the following concerning the parameters of the partial geometric designs obtained in Theorem \ref{th2}. If $D_{0}$ resp. $D_{1}$ are the quadratic residues resp. nonresidues modulo $q$, $A$ is the incidence matrix of the balanced incomplete block design given by $(\mathbb{F}_{q}, {\rm Dev}(D_{0})\cup{\rm Dev}(D_{1}))$, and $J$ is the all-one vector of length $q-1$, then the matrix $J\otimes A$ of length $q(q-1)$ is the incidence matrix of a partial geometric design with the same parameters as those obtained by Theorem \ref{th2}. This naively constructed design is simple and, according to our MAGMA computations, for small values of $q$ (ranging inclusively between $5$ and $25$), is isomorphic to that obtained by Theorem \ref{th2}. It seems that it is in general difficult to prove that they are isomorphic (or non-isomorphic), and this is left as an open problem.
\subsection{A family of balanced incomplete block designs}
Let ${\rm SA}_{2}(q)$ denote the special affine group of degree two over $\mathbb{F}_{q}$. Note that ${\rm SA}_{2}(q)$ can be identified with the semidirect product ${\rm SL}_{2}(q)\ltimes\mathbb{F}_{q}^{2}$ with group operation given by $(g,x)\cdot(g',y)=(gg',gy+x)$ for $(g,x),(g',y)\in {\rm SL}_{2}(q)\ltimes\mathbb{F}_{q}^{2}$. Group actions here will be given by $x^{(g,z)}=xg+z$ for $x\in \mathbb{F}_{q}^{2}$ and $(g,z)\in {\rm SA}_{2}(q)$, and we note that ${\rm SA}_{2}(q)$ is $2$-homogeneous on $\mathbb{F}_{q}^{2}$. Again let $D(q)=\{(\gamma^{i},\gamma^{j})\in \mathbb{F}_{q}^{2}\mid i+j\equiv 0 \ ({\rm mod} \ 2)\}$.
\begin{theorem}\label{th4.3} The incidence structure $(\mathbb{F}_{q}^{2},\lb D(q)\rb^{{\rm SA}_{2}(q)})$ is a $(q^{2},(q-1)^{2}/2,\lambda)$ design with $b$ blocks where \[
\lambda=\begin{cases}q\frac{q-1}{4}(\frac{(q-1)^{2}}{2}-1),&\text{if } q\equiv 1 \ ({\rm mod} \ 4),\\
                     q\frac{q-1}{2}(\frac{(q-1)^{2}}{2}-1),&\text{if } q\equiv 3 \ ({\rm mod} \ 4),\end{cases} \ \text{ and } \ b=\begin{cases}\frac{q^{4}+q^{3}}{2},&\text{if } q\equiv 1 \ ({\rm mod} \ 4),\\
                     q^{4}+q^{3},&\text{if } q\equiv 3 \ ({\rm mod} \ 4).\end{cases}\]
\end{theorem}
\begin{proof} Denote $D(q)$ by $D$. We will show that ${\rm SA}_{2}(q)_{D}={\rm SL}_{2}(q)_{D}\ltimes\{0\}$. Let $x=(\gamma^{i},\gamma^{j})\in D$. It is clear that $x^{(g,{\bf 0})}\in D$ if and only if $g\in {\rm SL}_{2}(q)_{D}$. If $(g,z)\in{\rm SA}_{2}(q)_{D}$ for some $g\notin {\rm SL}_{2}(q)_{D}$, then we must have $(g,{\bf 0})\in {\rm SA}_{2}(q)_{D}$ since stabilizers are subgroups. Thus ${\rm SA}_{2}(q)_{D}={\rm SL}_{2}(q)_{D}\ltimes H$ for some (additive) subgroup $H\leq \mathbb{F}_{q}^{2}$. Now suppose that $(g,z)\in {\rm SA}_{2}(q)_{D}$ for some $g\in {\rm SL}_{2}(q)_{(+)}$, and $z=(z_{1},z_{2})\ne {\bf 0}$ (we can, without loss of generality assume that $z_{2}\ne 0$). Then we must have that $(1,\gamma^{j})^{(g,z)}\in D$ for all $j \equiv 0 \ ({\rm mod} \ 2)$. Then \[
(1,\gamma^{j})^{(g,z)}=(1,\gamma^{j})\mb \gamma^{i_{1}} & 0\\ 0 & \gamma^{i_{2}} \me+(z_{1},z_{2})=(\gamma^{i_{1}}+z_{1},\gamma^{i_{2}+j}+z_{2})\in D.\] Then we must have that $\log_{\gamma}(z_{2}+\gamma^{i_{2}+j})\equiv \log_{\gamma}(z_{1}+\gamma^{i_{1}}) \ ({\rm mod} \ 2)$ for all $j \equiv 0 \ ({\rm mod} \ 2)$. But such a statement was already shown to be absurd in the proof of Theorem \ref{le4}. If $g\in {\rm SL}_{2}(q)_{(-)}$ (for the case where $q\equiv 1 \ ({\rm mod} \ 4)$), then we can reach a similar contradiction. The result then follows from Lemma \ref{le1} and the fact that ${\rm SA}_{2}(q)$ is $2$-homogeneous on $\mathbb{F}_{q}^{2}$.
\end{proof}
\begin{example} (i) The incidence structure $(\mathbb{F}_{5}^{2},\lb D(5)\rb^{{\rm SA}_{2}(q)})$ is a $(25,8,35)$ balanced incomplete block design with $375$ blocks. (ii) The incidence structure $(\mathbb{F}_{7}^{2},\lb D(7)\rb^{{\rm SA}_{2}(q)})$ is a $(49,18,357)$ balanced incomplete block design with $2744$ blocks.
\end{example}
We emphasize that it is possible to construct balanced incomplete block designs having the same parameters as those of Theorem \ref{th4.3}, but with repeated blocks. A new design can be obtained from a given design by taking the multiset union of an arbitrary numer of copies of the block set of the given design. Let $\lambda$ be defined as in Theorem \ref{th4.3}. Since $(q-1)^{2}/2-1$ divides $2\lambda/q$, and $\lambda\equiv 0 \ ({\rm mod} \ (q-1)^{2}/2-1)$, the construction method introduced in \cite{WIL} (also see Theorem 3.53 of \cite{STIN}) can produce a balanced incomplete block design with paramters $(q^{2},(q-1)^{2}/2,\lambda/q)$. Then taking $q$ copies of such a design would yield the same paramters as those obtained by Theorem \ref{th4.3}. However, the designs obtained from Theorem \ref{th4.3} are simple, and are therefore different than those constructed in this way which are not simple.
\section{Concluding Remarks}\label{sec5}
In this paper we investigated a new method for constructing partial geometric designs based on group actions. We used this new method to construct infinite families of partial geometric designs, some of which are new, from the actions of degree-two linear groups on certain subsets of $\mathbb{F}_{q}^{2}$ such as (but not limited to) the twin prime cyclotomy. Moreover, by computing the stabilizers of such subsets of $\mathbb{F}_{q}^{2}$, we also were able to construct a new family of balanced incomplete block designs. Some possible directions for further work include: (i) investigating the actions of degree-two linear groups on higher order standard cyclotomies (as were introduced in \cite{GUST}), (ii) computing the quantities $|S_{1}\cap S_{2}^{g}|$ for $S_{1},S_{2}\in \{D(q),D'(q)\}$, where $D(q)$ and $D'(q)$ are defined as in Section \ref{ssec4.1}, and where $g\in \left\{\mb a & b\\c & d \me\in {\rm SL}_{2}(q)\mid a,b,c,d\ne 0\right\}$ (see Remark \ref{re2}), and (iii) investigating the actions of matrix groups of dimension higher than two.

\section{Acknowledgment}

The authors are very grateful to the two anonymous reviewers for all of their detailed comments that greatly improved the quality and the presentation of this paper.

\bibliographystyle{plain}

\end{document}